\documentclass[10pt,a4paper]{article}

\usepackage[round]{natbib}
\usepackage{amsmath,amssymb,amsthm}
\usepackage{cancel}

\usepackage{graphicx}

\newtheorem{theorem}{Theorem}

\theoremstyle{definition}
\newtheorem{definition}{Definition}

\newtheorem{example}{Example}

\newtheorem*{fubini}{Theorem (Fubini-Tonelli)}

\DeclareMathOperator*{\I}{I}

\usepackage[left=3.7cm,right=3.7cm,top=3cm,bottom=3cm]{geometry}

\author{Janne V. Kujala\thanks{Address: Department of Mathematical
    Information Technology, University of Jyv\"askyl\"a, P.O.Box 35,
    FI-40014 Jyv\"askyl\"a, Finland.  Email address:
    \texttt{jvk@iki.fi}. }}

\title{A Remark on the Assumptions of Bayes' Theorem\thanks{This
    research was supported by the Academy of Finland (grant number
    121855). The author is grateful to Matti Vihola for comments.}}

\begin{document}

\maketitle

\begin{abstract}
  We formulate simple equivalent conditions for the validity
  of Bayes' formula for conditional densities.  We show that for any
  random variables $X$ and $Y$ (with values in arbitrary measurable
  spaces), the following are equivalent:
  \begin{enumerate}
    \item $X$ and $Y$ have a joint density w.r.t.\ a product measure
      $\mu\times\nu$,
    \item $P_{X,Y}\ll P_X\times P_Y$, (here $P_{\bullet}$ denotes the
    distribution of $\bullet$)
    \item $X$ has a conditional density $p(x\mid y)$ w.r.t.\ a
      $\sigma$-finite measure $\mu$,
    \item $X$ has a conditional distribution $P_{X\mid Y}$ such that
      $P_{X\mid y} \ll P_X$ for all $y$,
    \item $X$ has a conditional distribution $P_{X\mid Y}$ and a
      marginal density $p(x)$ w.r.t.\ a measure $\mu$ such that
      $P_{X\mid y}\ll\mu$ for all $y$.
  \end{enumerate}
  Furthermore, given random variables $X$ and $Y$ with a conditional
  density $p(y\mid x)$ w.r.t.\ $\nu$ and a marginal density $p(x)$
  w.r.t. $\mu$, we show that Bayes' formula
  \[
  p(x\mid y) = \frac{p(y\mid x)p(x)}{\int p(y\mid x)p(x)d\mu(x)}
  \]
  yields a conditional density $p(x\mid y)$ w.r.t.\ $\mu$ if and only
  if $X$ and $Y$ satisfy the above conditions.  Counterexamples
  illustrating the nontriviality of the results are given, and
  implications for sequential adaptive estimation are considered.  
  \bigskip
  
  \noindent
  \textbf{AMS2000 subject classifications:} 60A05; 60A10.
\end{abstract}

\section{Preliminaries}

Let $(\Omega,\mathcal{F},\Pr)$ be a probability space.  A \emph{random
variable} is a measurable mapping $X: \Omega\to \mathsf{X}$ to some
measurable space $(\mathsf{X},\mathcal X)$ (usually the real line
$\mathbb{R}$ equipped with the Borel $\sigma$-algebra
$\mathcal{B}({\mathbb{R}})$).  The \emph{distribution} of the random
variable is the measure $P_X: S\mapsto \Pr(X^{-1}(S))$ induced on
$\mathcal{X}$.  If $P_X(S) = \int_S p_X(x)d\mu(x)$ for all
$S\in\mathcal{X}$ for some measurable function
$p_X:\mathsf{X}\to[0,\infty]$ and some measure
$\mu:\mathcal{X}\to[0,\infty]$, then $p_X$ is called a \emph{density}
of $X$ w.r.t.\ $\mu$.  For brevity, we leave out the subscript of the
density when it matches the arguments, i.e, instead of $p_X(x)$, we
write simply $p(x)$.

We define the product
$\mu\times\nu:\mathcal{X}\otimes\mathcal{Y}\to[0,\infty]$ of arbitrary
measures $\mu:\mathcal{X}\to[0,\infty]$ and
$\nu:\mathcal{Y}\to[0,\infty]$ by
\[
S \mapsto \inf\left\{\,\sum_{k=1}^\infty \mu(A_k)\nu(B_k):
\{A_k\}_{k=1}^\infty\subset\mathcal{X},\
\{B_k\}_{k=1}^\infty\subset\mathcal{Y},\ S\subset \bigcup_{k=1}^\infty
A_k\times B_k\,\right\},
\]
where $\mathcal{X}\otimes\mathcal{Y}$ denotes the $\sigma$-algebra
generated by all measurable rectangles.

\begin{fubini}
  Suppose $(\mathsf{X},\mathcal{X},\mu)$ and
  $(\mathsf{Y},\mathcal{Y},\nu)$ are measure spaces and
  $f:\mathsf{X}\times\mathsf{Y}\to[-\infty,\infty]$ is a measurable
  function.  If either $f$ is integrable or $f$ is nonnegative with
  $\sigma$-finite support $\{\,(x,y):f(x,y)\ne0\,\}$, then
  \[
  \int f(x,y)d(\mu\times\nu)(x,y) =
  \int\left[\int f(x,y)d\mu(x)\right]d\mu(y) = 
  \int\left[\int f(x,y)d\nu(y)\right]d\mu(x).
  \]
\end{fubini}
\begin{proof}
  Follows from \citep{mukherjea1972}.
\end{proof}

If a pair of random variables
$(X,Y):\Omega\to\mathsf{X}\times\mathsf{Y}$ has a joint density
$p(x,y)$ w.r.t.\ $\mu\times\nu$, then we can apply Fubini's theorem to
write the marginal distributions as
\begin{eqnarray*}
  P_X(U) &=& P_{X,Y}(U\times \mathsf{Y}) = \int_U\left[\int p(x,y)d\nu(y)\right]d\mu(x),\\
  P_Y(V) &=& P_{X,Y}(\mathsf{X}\times V) = \int_V\left[\int p(x,y)d\mu(x)\right]d\nu(y),
\end{eqnarray*}
which implies that $p_X(x) = \int p(x,y)d\nu(y)$ and $p_Y(y) = \int
p(x,y)d\mu(x)$ are marginal densities w.r.t.\ $\mu$ and $\nu$,
respectively.

A \emph{transition measure} from $(\mathsf{Y},\mathcal{Y})$ to
$(\mathsf{X},\mathcal{X})$ is any function $\mu:
\mathsf{Y}\times\mathcal{X}\to[0,\infty]$ satisfying the following
axioms:
\begin{enumerate}
\item for every $y\in \mathsf{Y}$, the function $S\mapsto \mu(y,S)$ is a
measure on $\mathcal{X}$,
\item for every $S\in\mathcal{X}$, the function $y\mapsto \mu(y,S)$ is
$\mathcal{Y}$-measurable.
\end{enumerate}
The product of a transition measure
$\mu:\mathsf{Y}\times\mathcal{X}\to[0,\infty]$ and a $\sigma$-finite
measure $\nu:\mathcal{Y}\to[0,\infty]$ is given by
\[
(\mu\times\nu)(S) := \int \mu(y,S_y)d\nu(y)
\]
for all $S \in \mathcal{X}\otimes\mathcal{Y}$, where $S_y :=
\{\,x:(x,y)\in S\,\}$.  The product is a measure on
$\mathcal{X}\otimes\mathcal{Y}$.
%
If a transition measure $P_{X\mid Y}:
\mathsf{Y}\times\mathcal{X}\to[0,\infty]$ satisfying
\[
P_{X,Y} = P_{X\mid Y}\times P_Y
\]
exists, then it is called a \emph{conditional distribution} of $X$
given $Y$.  We will also use the shorthand $P_{X\mid y} := P_{X\mid
Y}(y, \cdot)$.  Note that a conditional distribution always exists for
a random variable in $(\mathbb{R}^n,\mathcal{B}(\mathbb{R}^n))$,
$(\mathbb{R}^\infty, \mathcal{B}(\mathbb{R}^\infty))$, or any other
complete separable metric space, but there are spaces where its
existence is not guaranteed \citep{shiryaev1996}.

If a conditional distribution $P_{X\mid Y}$ exists and satisfies
\begin{equation*}\label{eq:conddens}
P_{X\mid y}(S) = \int_S p(x\mid y)d\mu(x)
\end{equation*}
for all $S\in\mathcal{X}$, $y\in\mathsf{Y}$ for some measurable
nonnegative function $(x,y)\mapsto p(x\mid y)$ and some measure $\mu$,
then $p(x\mid y)$ is called a \emph{conditional density} of $X$ given
$y$.  If a joint density $p(x,y)$ exists w.r.t.\ $\mu\times\nu$, then
a conditional density can always be obtained by
\begin{equation*}\label{eq:jointcond}
p(x\mid y) := \begin{cases}
  \displaystyle\frac{p(x,y)}{p(y)},& p(y)>0,\\
  0, & p(y) = 0.
\end{cases}
\end{equation*}
(The value chosen for $p(y)=0$ is immaterial as the conditional
density is only determined $\mu\times P_Y$-a.e.)

\section{Regularity conditions for Bayesian estimation}
\label{sec:conds}

The following theorem gives a set of equivalent conditions under which
we can avoid the potential problems of nonexistent distributions or
densities.  
\begin{theorem}\label{xythm}
  Let $(X,Y):\Omega\to\mathsf{X}\times\mathsf{Y}$ be a pair of random
  variables.  Then, the following are equivalent:
  \begin{enumerate}
    \item $X$ and $Y$ have a joint density w.r.t.\ a product measure
      $\mu\times\nu$,
    \item $P_{X,Y}\ll P_X\times P_Y$,
    \item $X$ has a conditional density $p(x\mid y)$ w.r.t.\ a
      $\sigma$-finite measure $\mu$,
    \item $X$ has a conditional distribution $P_{X\mid Y}$ such that
      $P_{X\mid y} \ll P_X$ for all $y$,
    \item $X$ has a conditional distribution $P_{X\mid Y}$ and a
      marginal density $p(x)$ w.r.t.\ a measure $\mu$ such that
      $P_{X\mid y}\ll\mu$ for all $y$.
  \end{enumerate}
  Obviously the same conditions with the roles of $X$ and $Y$ reversed
  are also equivalent.  Furthermore,
  \begin{enumerate}
    \item[6.] if the above conditions hold for $X$ and $Y$, then they
      also hold for $X' = F(X)$ and $Y' = G(Y)$ where $F:
      \mathsf{X}\to\mathsf{X}'$ and $G:\mathsf{Y}\to\mathsf{Y}'$ are
      any measurable functions.
  \end{enumerate}
\end{theorem}
The conditions of the theorem are mild, being satisfied whenever
either $X$ or $Y$ is discrete as well as in most practical situations
with continuous random variables.  However, they preclude in
particular the following example:
\begin{example} \label{ex1}
Suppose that $X = Y \sim \mathrm{Uniform}[0,1]$.  The conditional
distribution $P_{X\mid y}(S) = [y\in S]$ is singular w.r.t.\ $P_X =
m_{[0,1]}$, where $m_{[0,1]}$ denotes the restriction of the Lebesgue
measure to $[0,1]$, and so condition 4 of Theorem~\ref{xythm} is not
satisfied.  The conditional density
\[
p(x\mid y) = [x = y] := \begin{cases}1,&x=y,\\0,&x\ne y\end{cases}
\]
exists w.r.t.\ the counting measure, but this measure is not
$\sigma$-finite and so this density does not satisfy condition 3.
Even though the joint distribution can be written as
\[
P_{X,Y}(S) = \int \left[\int_{S_y} [x=y] d\#(x)\right] dm_{[0,1]}(y),
\]
where $\#$ is the counting measure, the integrand $[x=y]$ does not
yield the joint density of condition~1 because the function $[x=y]$ is
not integrable w.r.t.\ $\#\times m_{[0,1]}$ and so Fubini's theorem
does not hold for the iterated integral.
\end{example}

\begin{example}
One interpretation of the conditions of Theorem~\ref{xythm} is given
by the fact that the Radon-Nikod\'ym derivative in the
measure-theoretic definition of mutual information
\[
\I(X;Y) = \int dP_{X,Y}\log\frac{dP_{X,Y}}{d(P_X\times P_Y)}
\]
exists precisely when $P_{X,Y}\ll P_X\times P_Y$ (condition 2).  In
case $P_{X,Y}$ is singular w.r.t.\ $P_X\times P_Y$,
\citet{kolmogorov1956} defines $\I(X;Y)=\infty$.  Thus, failure of the
conditions of Theorem~\ref{xythm} implies that observation of $Y$ is
expected to give an infinite amount of information about $X$ (and,
symmetrically, $X$ is expected to give an infinite amount of
information about $Y$).  In Example~\ref{ex1}, observation of $Y$
gives complete information about $X$ and this information is obviously
infinite (it would take an infinite number of bits on the average to
transmit the precise value of $X\sim\mathrm{Uniform}[0,1]$).  On the
other hand, if either $X$ or $Y$ has only a finite number of possible
values, then there is only a finite amount of information that can be
gained about it; this implies $\I(X;Y)<\infty$, and so condition 2 of
Theorem~\ref{xythm} is necessarily satisfied.
\end{example}

\subsection{Bayes' theorem}

The conditions of Theorem~\ref{xythm} are precisely those under which
Bayes' theorem can be applied to a conditional density:
\begin{theorem}\label{bayes}
  Let $(X,Y):\Omega\to\mathsf{X}\times\mathsf{Y}$ be a pair of random
  variables and suppose that $p(y\mid x)$ is a conditional density of
  $Y$ given $X$ w.r.t.\ to a measure $\nu$.  Then the following are
  equivalent:
  \begin{enumerate}
    \item[(a)] $X$ and $Y$ satisfy the conditions of
    Theorem~\ref{xythm}.
    \item[(b)] There exists a measurable subset $V\subset\mathsf{Y}$
      such that the measure $\nu'(B):=\nu(B\cap V)$ is
      $\sigma$-finite, $p(y\mid x)$ is a conditional density of $Y$
      given $X$ w.r.t.\ $\nu'$, and
      \[
      p(y) := \int p(y\mid x)dP_X(x)
      \] is a marginal density of $Y$ w.r.t.\ 
      $\nu'$.\footnote{\label{semifinite} If $\nu$ is
	\emph{semifinite} (for every nonnull $B\in\mathcal{Y}$ there
	exists $B'\subset B$ such that $0<\nu(B')<\infty$), then
	$p(y)$ is a density of $Y$ w.r.t.\ the original measure $\nu$,
	too.  
      }
    \item[(c)] Bayes' formula
      \begin{equation*}
	P_{X\mid y}(S) := \frac{\int_S p(y\mid x)dP_X(x)}
	{\int p(y\mid x)dP_X(x)}
      \end{equation*}
      defines a conditional distribution of $X$ given $Y$.
    \item[(c')] If $p(x)$ is a marginal density of $X$ w.r.t.\ a
      measure $\mu$, then 
      \begin{equation*}
	p(x\mid y) := \frac{p(y\mid x)p(x)}{\int p(y\mid x)p(x)d\mu(x)}
      \end{equation*}
      is a conditional density of $X$ given $Y$ w.r.t.~$\mu$.
  \end{enumerate}
\end{theorem}


The following example shows that in some pathological cases, it is
possible that (b) holds as stated above, but $p(y)$ is not a density
of $Y$ w.r.t.\ the original measure $\nu$.
\begin{example}
  Let $C\subset[0,1]$ be a meagre set with positive Lebesgue measure
  (e.g., a fat Cantor set) and define
  \[
  S := \{\,(x,y)\in[0,1]^2: x+y \in C \text{ or } x+y-1 \in C\,\}
  \]
  so that every section of $S$ is a cyclically shifted version of $C$.
  Let $P_X$ be the restriction of the Lebesgue measure to $[0,1]$, and
  define $P_{Y\mid x}$ through the conditional density $p(y\mid x) =
  [y\in S_x\cup\{0\}]$ w.r.t.\ the measure
  \[
  \nu(B) := [0\in B] + \begin{cases}
    0, &\text{$B$ meagre},\\
    \infty, &\text{otherwise}.
  \end{cases}
  \]
  As the meagre sets form a $\sigma$-ideal, this definition indeed
  yields a countably additive measure.  As every $S_x$ is meagre, we
  obtain
  \[
  \begin{split}
    P_{X,Y}(R) &= \int\left[\int p(y\mid x)d\nu(y)\right]dP_X(x) \\
               &= \int\nu(R_x\cap(S_x\cup\{0\}))dP_X(x)\\
               &= \int[0\in R_x]dP_X(x)\\
               &= P_X(\{\,x:(x,0)\in R\,\}),
  \end{split}
  \]
  which is a well-defined joint distribution (yielding $(X,Y)$
  uniformly distributed on $[0,1]\times\{0\}$) and satisfies the
  conditions of Theorem~\ref{xythm}.  However, the function
  \[
  p(y) := \int p(y\mid x)dP_x = 
  \begin{cases}
    P_X(S_y) = P_X(C), &y>0,\\
    1, &y=0,
  \end{cases}
  \]
  is not a density of $Y$ w.r.t.\ $\nu$, because
  \[
  \int p(y)d\nu(y) 
  = 1\cdot\nu(\{0\}) + P_X(C)\cdot\nu(\left]0,1\right]) = 
      1 + P_X(C)\cdot\infty = \infty.
  \]
  Nonetheless, in accordance with Theorem~\ref{bayes}(b), $p(y)$ is a
  density w.r.t.\ the restriction of $\nu$ to the $\sigma$-finite set
  $\{0\}$.
\end{example}

\subsection{Adaptive sequential estimation}

In adaptive sequential estimation (see, e.g.,
\citealp{mackay1992ch4,kujalalukka2006,bestvalue}), a random variable
$\Theta$ is estimated based on a sequence $y_{x_1},\dots,y_{x_T}$ of
independent (given $\theta$) realizations from some conditional
densities $p(y_{x_t}\mid\theta)$ indexed by trial \emph{placements}
$x_t$, each of which can be adaptively chosen from some set
$\mathsf{X_t}\subset\mathsf{X}$ based on the outcomes
$\{y_{x_1},\dots,y_{x_{t-1}}\}$ of the earlier observations.  The
placement decision function $d: \{y_{x_1},\dots,y_{x_{t-1}}\}\mapsto
x_t$ can be deterministic or random, and we also assume that there
exists a special placement value that signals the end of the
experiment.  Thus, the outcome $Y_d = \{Y_{X_1},\dots,Y_{X_T}\}$ of a
whole experiment governed by the decision function $d$ can be
considered as a single random variable with a random number $T$ of
components.  It is natural to ask the following question: \emph{under
  what conditions do the whole-experiment outcome $Y_d$ and $\Theta$
  satisfy the conditions of Theorem~\ref{xythm}?}

If $Y_x$ and $\Theta$ satisfy the conditions of Theorem~\ref{xythm}
for all $x$, then one can apply Bayes' formula to any \emph{finite}
set $\mathbf{y} = \{y_{x_t}\}_{t=1}^T$ of results sequentially:
\[
p(\theta\mid\mathbf{y}) \propto p(\theta)p(y_{x_1}\mid\theta)\cdots
p(y_{x_T}\mid\theta).
\]
This implies that $P_{\Theta\mid\mathbf{y}}\ll P_\Theta$ for all
$\mathbf{y}$ (condition 4) and as this condition makes no reference to
the distribution of $\mathbf{y}$, it follows that regardless of the
decision function $d$, the whole-experiment outcome variable $Y_d$ has
a joint density with $\Theta$ provided that the experiment terminates
with probability one (so that $\mathbf{y}$ is almost surely finite).
However, if there is a positive probability that the experiment does
not terminate, then it is possible that no joint density of $\Theta$
and $Y_d$ exists, even for constant placements:

\begin{example}\label{ex:infinite}
  Suppose that $X \sim \mathrm{Uniform}[0,1]$ and the random variables
  $Y_t\in\{0,1\}$ for $t=1,2,\dots$ are defined as a binary
  representation of $X$.  Then, although the conditional density
  $p(x\mid y_1,\dots,y_T)$ w.r.t.\ the Lebesgue measure is
  well-defined for any finite set of observations, the full sequence
  of results $Y := \{Y_t\}_{t=1}^\infty$ cannot have any joint density
  with $X$, because by condition 6 of Theorem~\ref{xythm}, that would
  imply that also the transformed variable
  \[ Y' := F(Y) := \sum_{t=1}^\infty 2^{-t} Y_t
  \] 
  would have a joint density with $X = Y'$, which contradicts the
  negative result of Example~\ref{ex1}.
\end{example}

\subsection{Proofs}
\begin{proof}[Proof of Theorem~\ref{xythm}]
\begin{description}
\item[2 $\Rightarrow$ 5:] Using the joint density $p := dP_{X,Y}
  /d(P_X\times P_Y)$, we obtain the induced marginal density $p(x)$
  w.r.t.\ the measure $P_X$ and the conditional density $p(x\mid y)$,
  which induces a conditional distribution $P_{X\mid y}\ll P_X$.
\item[5 $\Rightarrow$ 4:] Denoting $N := \{\,x\in\mathsf{X}:p(x) =
0\,\}$, we have
  \[ 
  0 = \int_N p(x)d\mu(x)=P_X(N)=\int P_{X\mid y}(N) dP_Y(y),
  \]
  which implies $P_{X\mid y}(N) = 0$ for $P_Y$-a.e.\ $y$.  However, as
  $P_{X\mid y}$ is only determined for $P_Y$-a.e.\ $y$, we are free to
  modify it so that $P_{X\mid y}(N) = 0$ for all $y$.  We will show
  that this $P_{X\mid y}$ is dominated by $P_X$ for all $y$.  Let
  $S\in\mathcal{X}$ be such that $P_X(S) = 0$.  Then, we have
  \[
  0 = P_X(S\setminus N) = \int_{S\setminus N} \underbrace{p(x)}_{>0}d\mu(x),
  \]
  which implies $\mu(S\setminus N) = 0$. As $P_{X\mid y}\ll\mu$, we
  have $P_{X\mid y}(S\setminus N) = 0$, but as also $P_{X\mid y}(N) =
  0$, we obtain $P_{X\mid y}(S) = 0$.  Thus, $P_{X\mid y}\ll P_X$ for
  all $y$.
\item[4 $\Rightarrow$ 3:] Choose $\mu = P_X$.
\item[3 $\Rightarrow$ 1:] By the definition of conditional density and
Fubini's theorem, we have
\[
P_{X,Y}(S) = \int\left[\int_{S_y}p(x\mid y)d\mu(x)\right]dP_Y(y) =
\int_S p(x\mid y)d(\mu\times P_Y)(x,y).
\]
Thus, $p(x\mid y)$ is a joint density of $X$ and $Y$ w.r.t.\
$\mu\times P_Y$.
\item[1 $\Rightarrow$ 2:] Suppose that $p(x,y)$ is a joint density
  w.r.t.\ $\mu\times\nu$ and let $S\in\mathcal{X}\otimes\mathcal{Y}$
  be an arbitrary measurable set such that $(P_X\times P_Y)(S) = 0$.
  We will show that then $P_{X,Y}(S) = 0$.  Denoting
  \begin{eqnarray*}
    U &:=& \{\,x\in \mathsf{X}: p(x) = 0\,\},\\
    V &:=& \{\,y\in \mathsf{Y}: p(y) = 0\,\},\\
    N &:=& (U\times \mathsf{Y})\cup (\mathsf{X}\times V),
  \end{eqnarray*}
  we have $P_X(U) = 0$ and $P_Y(V) = 0$.  Furthermore, as
  $\mu\times\nu$ is $\sigma$-finite on $S\setminus N$, Fubini's
  theorem yields
  \[
  0 = (P_X\times P_Y)(S\setminus N) = \int_{S\setminus N}
  \underbrace{p(x)}_{>0}\underbrace{p(y)}_{>0}d(\mu\times\nu)(x,y),
  \]
  which implies that $(\mu\times\nu)(S\setminus N) = 0$ and so
  $P_{X,Y}(S\setminus N) = 0$.  Thus,
  \[
  P_{X,Y}(S) \le P_{X,Y}(S\setminus N) + \underbrace{P_{X,Y}(U\times \mathsf{Y})}_{=P_X(U)} + \underbrace{P_{X,Y}(\mathsf{X}\times V)}_{=P_Y(V)} = 0.
  \]
\item[2 $\Rightarrow$ 6:] Suppose that $F: \mathsf{X}\to\mathsf{X}'$
  and $G:\mathsf{Y}\to\mathsf{Y'}$ are arbitrary measurable mappings.
  We show that $P_{X,Y}\ll P_X\times P_Y$ implies $P_{F(X),G(Y)}\ll
  P_{F(X)}\times P_{G(Y)}$.  For any $S\in\mathcal{X}\otimes\mathcal{Y}$,
  \[ 0 = (P_{F(X)}\times P_{G(Y)})(S) = (P_X\times
  P_Y)\left(\bigcup_{(x,y)\in S}F^{-1}(x)\times G^{-1}(y)\right)
  \]
  implies
  \[
  0 = P_{X,Y}\left(\bigcup_{(x,y)\in S}F^{-1}(x)\times G^{-1}(y)\right) = P_{F(X),G(Y)}(S),
  \]
  where $F^{-1}$ and $G^{-1}$ denote the preimage sets.\qedhere
\end{description}
\end{proof}

\begin{proof}[Proof of Theorem~\ref{bayes}]
  \textbf{(a) $\Rightarrow$ (b)} Denoting $S := \{\,(x,y)\in
  \mathsf{X}\times\mathsf{Y}:p(y\mid x)>0\,\}$, we obtain
  \[
  P_{X,Y}((\mathsf{X}\times\mathsf{Y})\setminus S) = 
  \int\left[\int_{\mathsf{Y}\setminus S_x}p(y\mid x)d\nu(y)\right]dP_X(x) = 0,
  \]
  which means that $S$ is a full set w.r.t.\ $P_{X,Y}$.  Denoting \[
  \begin{split}
  p(y) &:= \int p(y\mid x)dP_X(x),\\
  N &:= \{\,y\in\mathsf{Y}:p(y)=0\,\} = \{\,y\in\mathsf{Y}:P_X(S_y)=0\,\},
  \end{split}
  \]
  we have
  \[
  (P_X\times P_Y)(S\cap(\mathsf{X}\times N)) = \int_N P_X(S_y)dP_Y(y) = 0,
  \]
  and so the assumption $P_{X,Y}\ll P_X\times P_Y$ (condition 2)
  implies $P_{X,Y}(S\cap(\mathsf{X}\times N)) = 0$.

  Let $\mathcal{S}$ denote the class ($\sigma$-ideal) of all
  $V\in\mathcal{Y}$ such that $\nu$ is $\sigma$-finite on $V$.  Then,
  the supremum $M:=\sup_{V\in\mathcal{S}}\int_V p(y)d\nu(y)$ is
  obviously attained for some $V\in\mathcal{S}$, and for this $V$,
  Fubini's theorem yields
  \[
  \begin{split}
    M = \int_V p(y)d\nu(y) &= \int_V\left[\int p(y\mid x)dP_X(x)\right]d\nu(y)\\
    &= \int\left[\int_V p(y\mid x)d\nu(y)\right]dP_X(x) = P_Y(V),
  \end{split}
  \]
  implying that $M$ is finite.  As $M<\infty$ is the maximum value of
  the integral, we must have $\int_{B\setminus V}p(y)d\nu(y) = 0$ for
  any $B\in\mathcal{S}$, and so $\nu((B\setminus V)\setminus N) = 0$
  for any $B\in\mathcal{S}$.\footnote{This implies that $\nu(B)$ can
  only attain the values $0$ and $\infty$ for any measurable
  $B\subset\mathsf{Y}\setminus(V\cup N)$.  Hence, if $\nu$ is
  semifinite, only the value $0$ will be possible for these sets, and
  it follows that $\nu$ and $\nu'$ must agree on the support of
  $p(y)$.  This proves the statement of footnote~\ref{semifinite} on
  p.~\pageref{semifinite}.}  Thus, defining $\nu'(B) := \nu(B\cap V)$,
  we have $\nu'(B\setminus N) = \nu(B\setminus N)$ for any
  $B\in\mathcal{S}$.  As $p(y\mid x)$ is $\nu$-integrable for
  $P_X$-a.e.\ $x$, its support $S_x=\{\,y\in\mathcal{Y}:p(y\mid
  x)>0\,\}$ must belong to $\mathcal{S}$ for $P_X$-a.e.\ $x$.  It
  follows
  \[
  \begin{split}
  P_{X,Y}(R) &= P_{X,Y}((R\cap S)\setminus (\mathsf{X}\times N))\\
  &= \int\left[\int_{(R_x\cap S_x)\setminus N}p(y\mid x)d\nu(y)\right]dP_X(x)\\
  &= \int\left[\int_{(R_x\cap S_x)\setminus N}p(y\mid x)d\nu'(y)\right]dP_X(x)\\
  &\le \int\left[\int_{R_x}p(y\mid x)d\nu'(y)\right]dP_X(x)\\
  &\le \int\left[\int_{R_x}p(y\mid x)d\nu(y)\right]dP_X(x) = P_{X,Y}(R)
  \end{split}
  \]
  for all $R\in\mathcal{X}\otimes\mathcal{Y}$ and so $p(y\mid x)$ is a
  conditional density w.r.t.\ $\nu'$, too.  Furthermore, as $\nu'$ is
  $\sigma$-finite, Fubini's theorem yields
  \[
  \begin{split}
    \int_B p(y)d\nu'(y) &= \int_B\left[\int p(y\mid x)dP_X(x)\right]d\nu'(y)\\
    &= \int\left[\int_B p(y\mid x)d\nu'(y)\right]dP_X(x) = P_Y(B)
  \end{split}
  \]
  for all $B\in\mathcal{Y}$ and so $p(y)$ is a density of $Y$
  w.r.t.\ $\nu'$.

\textbf{(b) $\Rightarrow$ (c)} Assuming that $p(y)=\int p(y\mid
  x)dP_X(x)$ is a density of $Y$ w.r.t.\ a $\sigma$-finite measure
  $\nu'$, let us show that the the function $P_{X\mid y}$ defined by
  Bayes' formula is a well-defined conditional distribution.  Using
  the definitions and Fubini's theorem, we obtain
  \[
  \begin{split}
    \int \frac{\int_{S_y} p(y\mid x)dP_X(x)}{\int p(y\mid x)dP_X(x)}dP(y)
    &= \int\left[\int_{S_y} \frac{p(y\mid x)}{p(y)}dP_X(x)\right]dP_Y(y)\\
    &= \int\left[\int_{S_x} \frac{p(y\mid x)}{p(y)}dP_Y(y)\right]dP_X(x)\\
    &= \int\left[\int_{S_x} \frac{p(y\mid x)}{p(y)}p(y)d\nu'(y)\right]dP_X(x)\\
    &= \int\left[\int_{S_x}p(y\mid x)d\nu'(y)\right]dP_X(x) = P_{X,Y}(S).
  \end{split}
  \]

\textbf{(c') $\Leftrightarrow$ (c)} obvious.

\textbf{(c) $\Rightarrow$ (a)} As $P_{X\mid y}$ is given as an integral
  over $P_X$, condition 4 follows.
\end{proof}

\subsection{Generalization}

For completeness, we present a generalization of Theorem~\ref{xythm}
to more than two random variables.  To state the generalization, we
need another definition.
\begin{definition}
  A \emph{Bayes network} is a directed acyclic graph representing a
  dependency structure of a set $X_1,\dots,X_n$ of random
  variables. Each random variable $X_k$ is represented by a node whose
  parents are its conditioning variables
  $X_{j(k,1)},\dots,X_{j(k,n_k)}$, where we can assume WLOG that
  $j(k,i) < k$ for all $i=1,\dots,n_k$ (topological sorting), so that
  the joint distribution of $X_1,\dots,X_n$ is given by the product
  \[
  P_{X_1,\dots,X_n} = \prod_k P_{X_k\mid
    X_{j(k,1)},\dots,X_{j(k,n_k)}},
  \]
  where one can interpret, e.g., $P_{X_k\mid
  X_{j(k,1)},\dots,X_{j(k,n_k)}} = P_{X_k\mid X_1,\dots,X_{k-1}}$ and
  then apply the transition measure product operator.
\end{definition}

\begin{theorem}\label{xythm2}
  Let $X_1,\dots,X_n$ be random variables.  Then, the following are
  equivalent:
  \begin{enumerate}
    \item $X_1,\dots,X_n$ have a joint density,
    \item $P_{X_1,\dots,X_n}\ll P_{X_1}\times\dots\times P_{X_n}$,
    \item $P_{X_1,\dots,X_n}$ is representable as a Bayes network
        where each conditional distribution $P_{X_k\mid
        x_{j(k,1)},\dots,x_{j(k,n_k)}}$ has a density w.r.t.\ a
        $\sigma$-finite measure $\mu_k$,
    \item $P_{X_1,\dots,X_n}$ is representable as a Bayes network
      where each conditional distribution $P_{X_k\mid
      x_{j(k,1)},\dots,x_{j(k,n_k)}}$ is absolutely continuous
      w.r.t.~$P_{X_k}$.
    \item $P_{X_1,\dots,X_n}$ is representable as a Bayes network
      where each conditional distribution $P_{X_k\mid
      x_{j(k,1)},\dots,x_{j(k,n_k)}}$ is dominated by a measure
      $\mu_k$ w.r.t.\ which there exists a marginal density $p(x_k)$.
  \end{enumerate}
  Furthermore, 
  \begin{enumerate}
    \item[6.] if the above conditions hold for $X_1,\dots,X_n$, then
      they also hold for $X'_k = F_k(X_k)$, where $F_k:
      \mathsf{X}_k\to\mathsf{X}'_k$ are any measurable functions.
  \end{enumerate}
\end{theorem}
\begin{proof}
  \newcommand{\parents}[2]{{#1_{<#2}}}
  \newcommand{\ancestors}[2]{{#1_{:#2-1}}} 
  This proof is a straightforward generalization of the proof of
  Theorem~\ref{xythm}.

  For brevity, we shall denote the parents of $x_k$ by $\parents{x}{k}
  := (x_{j(k,1)},\dots,x_{j(k,n_k)})$.
\begin{description}
\item[2 $\Rightarrow$ 5:] The joint density $p := dP_{X_1,\dots,X_n}
  /d(P_{X_1}\times\dots\times P_{X_n})$ induces for each $k$ the
  conditional density $p(x_k\mid x_1,\dots,x_{k-1})$ w.r.t.\ the
  marginal distribution $P_{X_k}$.  Thus, the required Bayes network
  is given by $\parents{x}{k} := (x_1,\dots,x_{k-1})$ for all
  $k=1,\dots,n$.
\item[5 $\Rightarrow$ 4:] Let $k\in\{1,\dots,n\}$ be arbitrary.
  Denoting $N := \{\,x_k\in\mathsf{X}_k:p(x_k) = 0\,\}$, we have by
  the definition of conditional distribution
  \[
  0 = \int_N p(x_k)d\mu(x_k) = P_{X_k}(N)
  =\int_\parents{\mathsf{X}}{k}P_{X_k\mid\parents{x}{k}}(N)
  dP_\parents{X}{k}(\parents{x}{k}),
  \]
  which implies $P_{X_k\mid \parents{x}{k}}(N) = 0$ for
  $P_{\parents{X}{k}}$-a.e.\ $\parents{x}{k}$.  However, as
  $P_{X_k\mid \parents{x}{k}}$ is only determined for
  $P_{\parents{x}{k}}$-a.e.\ $\parents{x}{k}$, we are free to modify
  it so that $P_{X_k\mid \parents{x}{k}}(N) = 0$ for all
  $\parents{x}{k}$.  We will show that this $P_{X_k\mid
  \parents{x}{k}}$ is dominated by $P_{X_k}$ for all $\parents{x}{k}$.
  Let $S\in\mathcal{X}_k$ be such that $P_{X_k}(S) = 0$.  Then, we
  have
  \[
  0 = P_{X_{k}}(S\setminus N) = \int_{S\setminus N} \underbrace{p(x_k)}_{>0}d\mu(x_k),
  \]
  which implies $\mu_k(S\setminus N) = 0$. As $P_{X_k\mid
  \parents{x}{k}}\ll\mu_k$, we have $P_{X_k\mid
  \parents{x}{k}}(S\setminus N) = 0$, but as also $P_{X_k\mid
  \parents{x}{k}}(N) = 0$, we obtain $P_{X_k\mid \parents{x}{k}}(S) =
  0$.  Thus, $P_{X_k\mid \parents{x}{k}}\ll P_{X_k}$ for all
  $\parents{x}{k}$.

\item[4 $\Rightarrow$ 3:] Choose $\mu_k = P_{X_k}$.
\item[3 $\Rightarrow$ 1:] By the definition of the conditional
  densities and Fubini's theorem, we have
  \[
  \begin{split}
  P_{X_1,\dots,X_n}(S) &= \int_S \prod_k
  p(x_k\mid\parents{x}{k})d\mu_k(x_k)\\
  &=\int_S \left[\prod_k p(x_k\mid\parents{x}{k})\right] 
  d(\mu_1\times\dots\times\mu_n)(x).
  \end{split}
  \]
  Thus, $\prod_k p(x_k\mid\parents{x}{k})$ is a joint density of
  $X_1,\dots,X_n$ w.r.t.\ $\mu_1\times\dots\times\mu_k$.
\item[1 $\Rightarrow$ 2:] Suppose that $p(x_1,\dots,x_n)$ is a joint
  density w.r.t.\ $\mu_1\times\dots\times\mu_n$ and let
  $S\in\mathcal{X}_1\otimes\dots\otimes\mathcal{X}_n$ be an arbitrary
  measurable set such that $(P_{X_1}\times\dots\times P_{X_n})(S) =
  0$.  We will show that then $P_{X_1,\dots,X_n}(S) = 0$.  Denoting
  \begin{eqnarray*}
    N_k &:=& \{\,x_k\in \mathsf{X}_k: p(x_k) = 0\,\},\\ N &:=&
    \bigcup_k \mathsf{X}_1\times\dots\times\mathsf{X}_{k-1}\times
    N_k\times \mathsf{X}_{k+1}\times\dots\times\mathsf{X}_n,
  \end{eqnarray*}
  we have $P_{X_k}(N_k) = 0$ for all $k$. Furthermore, as
  $\mu_1\times\dots\times\mu_n$ is $\sigma$-finite on $S\setminus N$,
  Fubini's theorem yields
  \[
  \begin{split}
  0 &= (P_{X_1}\times\dots\times P_{X_n})(S\setminus N) = \int_{S\setminus N}
  \prod_k p(x_k)d\mu_k(x_k)\\
  &=\int_{S\setminus N}\underbrace{\left[\prod_k p(x_k)\right]}_{>0}d(\mu_1\times\dots\times\mu_k)(x),
  \end{split}
  \]
  which implies that $(\mu_1\times\dots\times\mu_n)(S\setminus N) = 0$
  and so $P_{X_1,\dots,X_n}(S\setminus N) = 0$.  Thus,
  \[
  P_{X_1,\dots,X_n}(S) \le P_{X_1,\dots,X_n}(S\setminus N) + \sum_k
  P_{X_k}(N_k) = 0.
  \]
\item[2 $\Rightarrow$ 6:] Suppose that $F_k:
  \mathsf{X}_k\to\mathsf{X}'_k$ are arbitrary measurable mappings.  We
  show that $P_{X_1,\dots,X_n}\ll P_{X_1}\times\dots\times P_{X_n}$
  implies $P_{F_1(X_1),\dots,F_n(X_n)}\ll P_{F_1(X_1)}\times\dots\times
  P_{F_n(X_n)}$.  For any $S\in\mathcal{X}_1\otimes\dots\otimes\mathcal{X}_n$,
  \[
  \begin{split}
  0 &= (P_{F_1(X_1)}\times\dots\times P_{F_n(X_n)})(S) \\
  &= (P_{X_1}\times\dots\times P_{X_n})\left(
  \bigcup_{x\in S}F_1^{-1}(x_1)\times\dots\times F_n^{-1}(x_n)
  \right)
  \end{split}
  \]
  implies
  \[
  0 = P_{X_1,\dots,X_n}
  \left(
  \bigcup_{x\in S}F_1^{-1}(x_1)\times\dots\times F_n^{-1}(x_n)
  \right) = P_{F_1(X_1),\dots,F_n(X_n)}(S),
  \]
  where $F_k^{-1}$ denotes the preimage set.\qedhere
\end{description}
\end{proof}

\end{document}